\documentclass[a4paper]{article}
\usepackage[a4paper,top=3cm,bottom=2cm,left=3cm,right=3cm,marginparwidth=1.75cm]{geometry}

\newcounter{saveenumi}
\newcommand{\seti}{\setcounter{saveenumi}{\value{enumi}}}
\newcommand{\conti}{\setcounter{enumi}{\value{saveenumi}}}

\usepackage{bm}
\usepackage{amsmath,amsthm}
\usepackage{amssymb}
\usepackage{graphicx}
\usepackage[colorinlistoftodos]{todonotes}
\usepackage[colorlinks=true, allcolors=blue]{hyperref}
\usepackage{xcolor}

\usepackage[english]{babel}
\usepackage[utf8x]{inputenc}
\usepackage[T1]{fontenc}

\newtheorem{definition}{Definition}[section]
\newtheorem{theorem}[definition]{Theorem}
\newtheorem{example}[definition]{Example}
\newtheorem{notation}[definition]{Notation}
\newtheorem{lemma}[definition]{Lemma}
\newtheorem{remark}[definition]{Remark}

\renewenvironment{proof}[1][\noindent Proof]{{\par\pushQED{\qed}\itshape #1\@. }}{\popQED}

\DeclareMathOperator{\im}{Im}
\DeclareMathOperator{\PG}{PG}
\DeclareMathOperator{\PGL}{PGL}
\newcommand{\qbin}[2]{\genfrac{[}{]}{0pt}{}{#1}{#2}}
 \newcommand{\qbreuk}[4]{\frac{(q^{#1}-1)\cdots (q^{#2}-1)}{(q^{#3}-1)\cdots (q^{#4}-1)}}
\newcommand{\qb}[2]{\frac{q^{#1}-1}{q^{#2}-1}}
\newcommand{\comments}[1]{}

\def\R{\mathbb{R}}

\title{Cameron-Liebler sets of $k$-spaces in $\PG(n,q)$}
\author{A. Blokhuis\footnote{Address: 
Department of Mathematics and Computing Science, 
Eindhoven University of Technology,
P.O. Box 513,
5600 MB Eindhoven,
the Netherlands\newline Email address: \{a.blokhuis,m.de.boeck\}@tue.nl, Website: http://www.win.tue.nl/$\sim$aartb/ }, M. De Boeck\footnotemark[1], J. D'haeseleer\footnote{Address: 
Department of Mathematics: Analysis, Logic and Discrete Mathematics, 
UGent, Krijgslaan 281 -- S8, 9000 Gent, Flanders, Belgium\newline Email address: jozefien.dhaeseleer@ugent.be }}
\date{}
\begin{document}
\maketitle

\begin{abstract}
Cameron-Liebler sets of $k$-spaces were introduced recently in \cite{Ferdinand.}. We list several equivalent definitions for these Cameron-Liebler sets, by making a generalization of known results about Cameron-Liebler line sets in $\PG(n,q)$ and Cameron-Liebler sets of $k$-spaces in $\PG(2k+1,q)$. We also present some classification results. \\
\end{abstract}

\textbf{Keywords}: Cameron-Liebler set, Grassmann graph.
\par 
\textbf{MSC 2010 codes}:  05B25, 51E20, 05E30, 51E14, 51E30.

\section{Introduction}
In \cite{CL} Cameron and Liebler introduced specific line classes in $\PG(3,q)$ when investigating the orbits of the projective groups $\PGL(n+1,q)$. These line sets $\mathcal{L}$ have the property that every line spread $\mathcal{S}$ in $\PG(3,q)$ has the same number of lines in common with $\mathcal{L}$. A lot of equivalent definitions for these sets of lines are known. An overview of the equivalent definitions can be found in \cite[Theorem $3.2$]{phdDrudge}.

After a large number of results regarding Cameron-Liebler sets of lines in the projective space $\PG(3, q)$, Cameron-Liebler sets of $k$-spaces in $\PG(2k+1, q)$ \cite{CLkclas}, and Cameron-Liebler line sets in $\PG(n, q)$ \cite{phdDrudge} were defined. In addition, this research started the motivation for defining and investigating Cameron-Liebler sets of generators in polar spaces \cite{CLpol} and Cameron-Liebler classes in finite sets \cite{CLset}. In fact Cameron-Liebler sets can be introduced for any distance-regular graph. This has been done in the past under various names: boolean degree $1$ functions, completely regular codes of strength $0$ and covering radius $1$, ... We refer to the introduction of \cite{Ferdinand.} for an overview. Note that the definitions do not always coincide, e.g. for polar spaces.

One of the main reasons for studying Cameron-Liebler sets is that there are several equivalent definitions for them, some algebraic, some geometrical (combinatorial) in nature. In this paper we investigate Cameron-Liebler sets of $k$-spaces in $\PG(n,q)$. In Section $2$ we give several equivalent definitions for these Cameron-Liebler sets of $k$-spaces. Several properties of these Cameron-Liebler sets are given in the third section.

The main question, independent of the context where Cameron-Liebler sets are investigated, is always the same: for which values of the parameter $x$ do there exist Cameron-Liebler sets and which examples correspond to a given parameter $x$.

For the Cameron-Liebler line sets, classification results and non-trivial examples were discussed in \cite{CL6,CL,CL10,phdDrudge,CL19,CL191,CL20,CL21,CL22,CL25,CL26,CL33}. The strongest classification results are given in \cite{CL21,CL26}, the latter of which proves that there exists a constant $c>0$ so that there are no Cameron-Liebler line sets in $\PG(3,q)$ with parameter $2<x<cq^{4/3}$. In \cite{CL6,new,CL10,phdDrudge,CL19,CL20} the constructions of two non-trivial Cameron-Liebler line sets with parameter $x=\frac{q^2+1}{2}$ and $x=\frac{q^2-1}{2}$ were given. Classification results for Cameron-Liebler sets of generators in polar spaces were given in \cite{CLpol} and for Cameron-Liebler classes of sets, a complete classification was given in \cite{CLset}. Regarding the Cameron-Liebler sets of $k$-spaces in $\PG(2k+1,q)$, the classification results are described in \cite{Klaus,CLkclas}.

If $q \in \{2,3,4,5\}$ a complete classification is known for Cameron-Liebler sets of $k$-spaces in $\PG(n,q)$, see \cite{Ferdinand.}. There the authors show that the only Cameron-Liebler sets in this context are the trivial Cameron-Liebler sets.
In the last section, we use the properties from Section \ref{sec3} to give the following classification result: there is no Cameron-Liebler set of $k$-spaces in $\PG(n,q)$, $n>3k+1$, with parameter $x$ such that $2\leq x\leq\frac{1}{\sqrt[8]{2}} q^{\frac{n}{2}-\frac{k^2}{4}-\frac{3k}{4}-\frac{3}{2}}(q-1)^{\frac{k^2}{4}-\frac{k}{4}+\frac{1}{2}}\sqrt{q^2+q+1}$.

\section{The characterization theorem}
Note first that we will always work with projective dimensions and that vectors are regarded as column vectors. Let $\Pi_k$ be the collection of $k$-spaces in $\PG(n,q)$ for $0 \leq k \leq n$ and let $A$ be the incidence matrix of the points and the $k$-spaces of $\PG(n,q)$: the rows of $A$ are indexed by the points and the columns by the $k$-spaces. 

We define $A_i$ as the adjacency matrix of the relation $R_i$ with \[R_i=\{ (\pi,\pi')| \pi,\pi' \in \Pi_k, \dim(\pi \cap \pi') = k-i\}, \quad  0\leq i\leq k+1.\] These relations $R_0, R_1, \dots, R_{k+1}$ form the Grassmann association scheme $J_q(n+1,k+1)$. Remark that $A_0 = I$ and $\sum_{i=0}^{k+1} A_i = J$ where $I$ and $J$ are the identity matrix and all-one matrix respectively. We denote the all-one vector by $\bm{j}$. Note that the Grassmann graph for $k$-spaces in $\PG(n,q)$ has adjacency matrix $A_1$. 

It is known that there is an orthogonal decomposition $V_0 \perp V_1 \perp \dots \perp V_{k+1}$ of $\mathbb{R}^{\Pi_k}$ in maximal common eigenspaces of $A_0,A_1,\dots,  A_{k+1}$. In the following lemmas and theorems, we denote the disjointness matrix $A_{k+1}$ also by $K$ since the corresponding graph is a $q$-Kneser graph. For more information about the Grassmann schemes we refer to \cite[Section $9.3$]{BCN} and \cite[Section $9$]{meagen}. \\

We will use the \emph{Gaussian binomial coefficient} ${\qbin ab}_q$ for $a,b\in \mathbb{N}$ with $a\geq b>0$, and prime power $q \geq 2$:
\begin{align*}
{\qbin ab}_q = \qbreuk{a}{a-b+1}{b}{}.
\end{align*}
The Gaussian binomial coefficient ${\qbin ab}_q$ is equal to the number of $b$-spaces of the vector space $\mathbb{F}^a_q$, or in the projective context, the number of $(b-1)$-spaces in the projective space $\PG(a-1,q)$. If the field size $q$ is clear from the context, we will write ${\qbin ab}$ instead of ${\qbin ab}_q$. We set $\qbin{a}{0}=1$ for all $a\in \mathbb{N}$; this is indeed the number of 0-spaces of a vector space.

The following counting result will be used several times in this article.
\begin{lemma}[{\cite[Section 170]{Segre}}]\label{lemmadisjunct}
The number of $j$-spaces disjoint from a fixed $m$-space in $\PG(n,q)$ equals $q^{(m+1)(j+1)}\qbin{n-m}{j+1}$.
\end{lemma}

To end the introduction of this section, we give the definition of a $k$-spread and a partial $k$-spread of $\PG(n,q)$.

\begin{definition}
A \emph{partial $k$-spread} of $\PG(n,q)$ is a collection of $k$-spaces which are mutually disjoint. A \emph{$k$-spread} in $\PG(n,q)$ is a partial $k$-spread in $\PG(n,q)$ that partitions the point set of $\PG(n,q)$. 
\end{definition}

Remark that a $k$-spread of $\PG(n,q)$ exists if and only if $k+1$ divides $n+1$, and necessarily contains $\frac{q^{n+1}-1}{q^{k+1}-1}$ elements (\cite{segre2}). 
A regular $k$-spread is a $k$-spread that can be constructed using field reduction. 

Before we start with proving some equivalent definitions for a Cameron-Liebler set of $k$-spaces, we give some lemmas and definitions that we will need in the characterization Theorem \ref{theodef}.

\begin{lemma}[{\cite{Delsarte}}]\label{eigenvallem}
Consider the Grassmann scheme $J_q(n+1,k+1)$. The eigenvalue $P_{ji}$ of the distance-$i$ relation for $V_j$ is given by:
\begin{align*}\label{eigenval}
P_{ji} = \sum\limits_{s=\max{(0,j-i)}}^{\min{(j,k+1-i)}} (-1)^{j+s}\qbin{j}{s}\begin{bmatrix}n-k+s-j \\ n-k-i\end{bmatrix} \begin{bmatrix} k+1-s \\ i\end{bmatrix} q^{i(i+s-j)+\binom{j-s}{2}}.
\end{align*}
\end{lemma}


\begin{lemma} \label{lemma2}
If $P_{1i}, i\geq 1,$ is the eigenvalue of $A_i$ corresponding to $V_j$, then $j=1$.
\end{lemma}
\begin{proof}
We need to prove that $P_{1i} \neq P_{ji}$ for $q$ a prime power and $j>1$. We will first introduce $\phi_i(j) = \max\left\{a\mid q^a|P_{ji} \right\}$, which is the exponent of $q$ in the factorization of $P_{ji}$. Remark that $\qbin a b$ equals $1$ modulo $q$ and note that it is sufficient to show that $\phi_i(j)$, $j>1$,  is different from $\phi_i(1)$ for all $i$. By Lemma \ref{eigenvallem} we see that \[{\phi_i(j)}= \min\left\{i(i+s-j)+\binom{j-s}{2}\mid \max\{0,j-i\} \leq s \leq \min\{j,k+1-i\}\right\}\] unless there are two or more terms with a power of $q$ with minimal exponent as factor and that have zero as their sum. 
If $s$ is the integer in $\{ \max\{0,j-i\},\dots, \min\{j,k+1-i\}\}$ closest to $j-i-\frac{1}{2}$, then $f_{ij}(s)=i(i+s-j)+\binom{j-s}{2}$ is minimal.
\begin{itemize}
\item If $j \leq i$, we see that $f_{ij}(s)$ is minimal for $s=0$. Then we find ${\phi_i(j)}=\frac 12 j^2-(i+\frac{1}{2})j+i^2$. We see that for a fixed $i$, $\phi_i(k-1)>\phi_i(k), k\leq i$. Note that the minimal value for $f_{ij}(s)$ is reached for only one $s$.
\item If $j \geq i$, we see that $f_{ij}(s)$ is minimal for $s=j-i$. Then we find ${\phi_i(j)}={\binom{i}{2}}$. Again we note that the minimal value for $f_{ij}(s)$ is reached for only one $s$.
\end{itemize}
We can conclude the following inequality for a given $i\geq 1$:
\begin{align*}
\phi_i(1)>\phi_i(2)>\dots >\phi_i(i)=\phi_i(i+1)=\dots=\phi_i(k+1) \, .
\end{align*}
This implies the statement for $i\neq 1$.

For $i = 1$ we have $P_{11}=-\qbin{k+1}{1}+\qbin{n-k}{1}\qbin{k}{1}q$ and $P_{j1}=-\qbin{j}{1}\qbin{k-j+2}{1}+\qbin{n-k}{1}\qbin{k+1-j}{1}q$, so we can see that they are different if $j \neq n+1$. This is always true since $j\in \{1,\dots ,k+1\}$ and $k<n$.
\end{proof}

Note that for $j\geq 1$ it was already known that $|P_{ji}| \leq |P_{1i}|$. This result was shown in \cite[Proposition $5.4(ii)$]{brouwer}.

\begin{lemma}\label{lemmaaantaldisjunct}
Let $\pi$ be a $k$-dimensional subspace in $\PG(n,q)$ with $\chi_\pi$ the characteristic vector of the set $\{\pi \}$. If $\mathcal{Z}$ is the set of all $k$-spaces in $\PG(n,q)$ disjoint from $\pi$ with characteristic vector $\chi_\mathcal{Z}^{}$, then
\begin{align*}
\chi_\mathcal{Z}^{} -q^{k^2+k}\begin{bmatrix}n-k-1\\k\end{bmatrix} \left(\begin{bmatrix}
n\\k\end{bmatrix}^{-1}  \bm{j} -\chi_\pi \right) \in \ker(A).
\end{align*}
\end{lemma}
\begin{proof}
Let $v_\pi$ be the incidence vector of $\pi$ with its positions corresponding to the points of $\PG(n,q)$. Note that $A\chi_\pi = v_\pi$. We have that $A\chi_\mathcal{Z}^{} = q^{k^2+k}\qbin{n-k-1}{k}(\bm{j}-v_\pi)$ since $\mathcal{Z}$ is the set of all $k$-spaces disjoint from $\pi$ and every point not in $\pi$ is contained in $q^{k^2+k}\qbin{n-k-1}{k}$ $k$-spaces skew to $\pi$ (see Lemma \ref{lemmadisjunct}). The lemma now follows from
\begin{align*}
&\chi_\mathcal{Z}^{} -q^{k^2+k}\begin{bmatrix}n-k-1\\k\end{bmatrix} \left(\begin{bmatrix}
n\\k\end{bmatrix}^{-1}  \bm{j} -\chi_\pi \right) \in \ker(A)\\
\Leftrightarrow\quad & A\chi_\mathcal{Z}^{} =q^{k^2+k}\begin{bmatrix}n-k-1\\k\end{bmatrix} \left(\begin{bmatrix}
n\\k\end{bmatrix}^{-1} A\bm{j} -A\chi_\pi \right)\;. \qedhere
\end{align*}
\end{proof}

\begin{definition}
A \emph{switching set} is a partial $k$-spread $\mathcal{R}$ for which there exists a partial $k$-spread $\mathcal{R}'$ such that $\mathcal{R} \cap \mathcal{R}' = \emptyset$, and $\cup_{\pi\in \mathcal{R}} \pi = \cup_{\pi\in \mathcal{R}'} \pi$, in other words, $\mathcal{R}$ and $\mathcal{R}'$ have no common members and cover the same set of points. We say that $\mathcal{R}$ and $\mathcal{R}'$ are a \emph{pair of conjugate switching sets}.
\end{definition}

The next lemma is a classical result in design theory.
\begin{lemma}\label{lemmaAfullrow}
The incidence matrix of a $2$-design has full row rank. 
\end{lemma}


The following lemma gives the relation between the common eigenspaces $V_0$ and $V_1$ of the matrices $A_i,i \in \{0,\dots, k+1\}$, and the row space of the matrix $A$. For the proof we refer to \cite[Theorem 9.1.4]{meagen}.

\begin{lemma} \label{lemmaAgelijkaanV_0V_1}
For the Grassmann scheme $J_q(n+1,k+1)$ we have that $\im(A^T)=V_0 \perp V_1$ and $V_0 = \langle \bm{j} \rangle$.
\end{lemma}

We want to make a combination of a generalization of Theorem $3.2$ in \cite{phdDrudge} and Theorem $3.7$ in \cite{CLkclas} to give several equivalent definitions for a Cameron-Liebler set of $k$-spaces in $\PG(n,q)$.
\begin{theorem} \label{theodef}
Let $\mathcal{L}$ be a non-empty set of $k$-spaces in $\PG(n,q), n\geq 2k+1$, with characteristic vector $\chi$, and $x$ so that $|\mathcal{L}|=x\begin{bmatrix}
n\\k\end{bmatrix}$. Then the following properties are equivalent.
\begin{enumerate}
\item $\chi \in$ $\im(A^T)$.
\item $\chi \in \ker(A)^\perp$.
\item  For every $k$-space $\pi$, the number of elements of $\mathcal{L}$ disjoint from $\pi$ is $(x-\chi(\pi))\qbin{n-k-1}{k}q^{k^2+k}$.
\item The vector $\chi -x\frac{q^{k+1}-1}{q^{n+1}-1} \bm{j}$ is a vector in $V_1$. 
\item $\chi \in V_0 \perp V_1$.
\item  For a given $i\in \{1, \dots, k+1\}$ and any $k$-space $\pi$, the number of elements of $\mathcal{L}$, meeting $\pi$ in a $(k-i)$-space is given by:
\begin{align*}
\begin{cases}
 \left( (x-1) \frac{q^{k+1}-1}{q^{k-i+1}-1}+q^{i}\frac{q^{n-k}-1}{q^{i}-1}\right) q^{i(i-1)}\begin{bmatrix}
n-k-1 \\ i-1 \end{bmatrix} \begin{bmatrix} k\\i \end{bmatrix}& \mbox{if } \pi \in \mathcal{L}\\  x \begin{bmatrix} n-k-1 \\ i-1\end{bmatrix} \begin{bmatrix} k+1 \\ i \end{bmatrix}q^{i(i-1)} & \mbox{if }\pi \notin \mathcal{L}
\end{cases}\;. 
\end{align*}
\item for every pair of conjugate switching sets $\mathcal{R}$ and $\mathcal{R}'$, we have that $|\mathcal{L} \cap \mathcal{R}| = |\mathcal{L} \cap \mathcal{R}'|$.
\seti
\end{enumerate}
If $\PG(n,q)$ admits a $k$-spread, then the following properties are equivalent to the previous ones.
\begin{enumerate} \conti
\item $|\mathcal{L}\cap \mathcal{S}| = x$ for every $k$-spread $\mathcal{S}$ in $\PG(n,q)$.
\item $|\mathcal{L}\cap \mathcal{S}| = x$ for every regular $k$-spread $\mathcal{S}$ in $\PG(n,q)$.
\end{enumerate}
\end{theorem}
\begin{proof}
We first prove that properties $1,2,3,4,5,6$ are equivalent by  proving the following implications:
\begin{itemize}
\item $1 \Leftrightarrow 2$: This follows since $\im(B^T)$ $= \ker(B)^\perp$ for every matrix $B$.

\item $2 \Rightarrow 3$: We assume that $\chi \in \ker(A)^\perp$. Let $\pi \in \Pi_k$ and $\mathcal{Z}$ the set of $k$-spaces disjoint from $\pi$. By Lemma \ref{lemmaaantaldisjunct}, we know that \begin{align*}
&\chi_\mathcal{Z}^{} -q^{k^2+k}\begin{bmatrix}n-k-1\\k\end{bmatrix} \left(\begin{bmatrix}
n\\k\end{bmatrix}^{-1} \bm{j} -\chi_\pi \right) \in \ker(A).
\end{align*}
Since $\chi \in \ker(A)^\perp$, this implies
\begin{align*}
 & \ \chi_\mathcal{Z}^{} \cdot \chi -q^{k^2+k}\begin{bmatrix}n-k-1\\k\end{bmatrix} \left(\begin{bmatrix}
n\\k\end{bmatrix}^{-1} \bm{j} \cdot \chi -\chi_\pi \cdot \chi \right)=0 \\
\Leftrightarrow & \ |\mathcal{Z}\cap \mathcal{L}| -q^{k^2+k}\begin{bmatrix}n-k-1\\k\end{bmatrix} \left(\begin{bmatrix}
n\\k\end{bmatrix}^{-1} |\mathcal{L}| - \chi(\pi) \right)=0 \\
\Leftrightarrow & \ |\mathcal{Z}\cap \mathcal{L}| =(x-\chi(\pi)) q^{k^2+k}\begin{bmatrix}n-k-1\\k\end{bmatrix}\;.
\end{align*}
The last equality shows that the number of elements of $\mathcal{L}$ disjoint from $\pi$ is $(x-\chi(\pi)) q^{k^2+k}\qbin{n-k-1}{k}$.
\item $3 \Rightarrow 4$: By expressing property $3$ in vector notation, we find that $K\chi = (x\bm{j}-\chi)\qbin{n-k-1}{k}q^{k^2+k}$ and since by Lemma \ref{lemmadisjunct} we have $K\bm{j} = q^{(k+1)^2}\qbin{n-k}{k+1}$, we see that $v=\chi -x\frac{q^{k+1}-1}{q^{n+1}-1} \bm{j}$ is an eigenvector of $K$:
\begin{align*}
Kv=& \ K\left(\chi -x\frac{q^{k+1}-1}{q^{n+1}-1} \bm{j} \right) \\
=& \ (x \bm{j} - \chi)\qbin{n-k-1}{k} q^{k^2+k} - x\frac{q^{k+1}-1}{q^{n+1}-1}q^{(k+1)^2}\begin{bmatrix} n-k \\ k+1 \end{bmatrix} \bm{j}\\ 
=& \ \qbin{n-k-1}{k}q^{k^2+k} \left( x \bm{j} -\chi -x \frac{q^{n+1}-q^{k+1}}{q^{n+1}-1} \bm{j} \right)\\
=& \ -\qbin{n-k-1}{k}q^{k^2+k} \left( \chi -x \frac{q^{k+1}-1}{q^{n+1}-1} \bm{j} \right) \\
=& \ P_{1,k+1} v\;.
\end{align*} 
By Lemma \ref{lemma2} for $i=k+1$, we know that $v\in V_1$. 
\item $4 \Rightarrow 5$: This follows since $V_0=\langle \bm{j} \rangle$ (see Lemma \ref{lemmaAgelijkaanV_0V_1}).
\item $5 \Rightarrow 1$: This follows from Lemma \ref{lemmaAgelijkaanV_0V_1}.
\end{itemize}


\begin{itemize}
\item $4 \Rightarrow 6$: Denote $\chi - x\frac{q^{k+1}-1}{q^{n+1}-1} \bm{j}$ by $v$.
The matrix $A_i$ corresponds to the relation $R_i$.
This implies that $(A_i \chi)_\pi$ gives the number of $k$-spaces in $\mathcal{L}$ that intersect $\pi$ in a $(k-i)$-space.
\begin{align*}
A_i \chi &= A_iv+x\frac{q^{k+1}-1}{q^{n+1}-1}A_i \bm{j}= P_{1i}v+x\frac{q^{k+1}-1}{q^{n+1}-1}P_{0i}\bm{j} \\
&= \left( - \begin{bmatrix}n-k-1 \\i-1 \end{bmatrix}\begin{bmatrix}k+1 \\ i\end{bmatrix}q^{i(i-1)}+\begin{bmatrix}n-k \\ i\end{bmatrix}\begin{bmatrix}k \\i \end{bmatrix}q^{i^2} \right) \left(\chi - x\frac{q^{k+1}-1}{q^{n+1}-1} \bm{j}\right) \\&+ x\frac{q^{k+1}-1}{q^{n+1}-1} \begin{bmatrix}n-k\\i\end{bmatrix} \begin{bmatrix}k+1 \\ i\end{bmatrix} q^{i^2} \bm{j}\\
&= \left( \begin{bmatrix}n-k \\ i\end{bmatrix}\begin{bmatrix} k\\i \end{bmatrix}q^{i^2}-\begin{bmatrix}k+1 \\ i\end{bmatrix}\begin{bmatrix}
n-k-1 \\i-1 \end{bmatrix}q^{i(i-1)} \right) \chi \\
&+ x\qb{k+1}{n+1}q^{i(i-1)}\left( 
\qbin{n-k-1}{i-1}\qbin{k+1}{i}-\qbin{n-k}{i}\qbin{k}{i}q^i + \qbin{n-k}{i}\qbin{k+1}{i}q^i \right) \bm{j}\\
&= \left( \begin{bmatrix}n-k \\ i\end{bmatrix}\begin{bmatrix} k\\i \end{bmatrix}q^{i^2}-\begin{bmatrix}k+1 \\ i\end{bmatrix}\begin{bmatrix}
n-k-1 \\i-1 \end{bmatrix}q^{i(i-1)} \right) \chi \\
&+ x\qb{k+1}{n+1}q^{i(i-1)}\qbin{n-k-1}{i-1} \qbin{k}{i} \left( 
\qb{k+1}{k-i+1}-\qb{n-k}{i}q^i\left(1-\qb{k+1}{k-i+1}\right) \right) \bm{j}\\
&= \left( \begin{bmatrix}n-k \\ i\end{bmatrix}\begin{bmatrix} k\\i \end{bmatrix}q^{i^2}-\begin{bmatrix}k+1 \\ i\end{bmatrix}\begin{bmatrix}
n-k-1 \\i-1 \end{bmatrix}q^{i(i-1)} \right) \chi + x\begin{bmatrix}n-k-1 \\ i-1\end{bmatrix}\begin{bmatrix}k+1\\i\end{bmatrix}q^{i(i-1)}\bm{j}
\end{align*} Remark that this proves the implication for every $i \in \{1,\dots, k+1\}$.

\item $6 \Rightarrow 4$: We follow the approach of Lemma $3.5$ in \cite{CLkclas} where we look for an eigenvalue of $A_i$ and we define $\beta_i = x\qbin{k+1}{i} \qbin{n-k-1}{i-1}q^{i(i-1)}$. \\
From property $6$ we know that
\begin{align*}
A_i \chi &= x\begin{bmatrix}k+1 \\ i\end{bmatrix}\begin{bmatrix}n-k-1\\i-1\end{bmatrix}q^{i(i-1)} (\bm{j}-\chi)  \\ &+ \left( (x-1) \frac{q^{k+1}-1}{q^{k-i+1}-1}+q^{i}\frac{q^{n-k}-1}{q^{i}-1}\right) q^{i(i-1)}\begin{bmatrix}
n-k-1 \\ i-1 \end{bmatrix} \begin{bmatrix} k\\i \end{bmatrix} \chi \\
&=\left( \begin{bmatrix}n-k \\ i\end{bmatrix}\begin{bmatrix} k\\i \end{bmatrix}q^{i^2}-\begin{bmatrix}k+1 \\ i\end{bmatrix}\begin{bmatrix}
n-k-1 \\i-1 \end{bmatrix}q^{i(i-1)} \right) \chi + x\begin{bmatrix}n-k-1 \\ i-1\end{bmatrix}\begin{bmatrix}k+1\\i\end{bmatrix}q^{i(i-1)}\bm{j}\\
&=P_{1i} \chi + \beta_i \bm{j}\;.
\end{align*}
Then we can see that $v_i = \chi + \frac{\beta_i}{P_{1i} - P_{0i}}\bm{j}$ is an eigenvector for $A_i$ with eigenvalue $P_{1i}$:
\begin{align*}
 A_i\left(\chi + \frac{\beta_i}{P_{1i} - P_{0i}}\bm{j}\right) =& P_{1i} \chi + \beta_i \bm{j} + \frac{\beta_i}{P_{1i} - P_{0i}}P_{0i} \bm{j} \\
 =& P_{1i} \left(\chi + \frac{\beta_i}{P_{1i} - P_{0i}}\bm{j}\right).
 \end{align*}
By Lemma \ref{lemma2} we know that $\chi + \frac{\beta_i}{P_{1i} - P_{0i}}\bm{j} = \chi -x\frac{q^{k+1}-1}{q^{n+1}-1} \bm{j}   \in V_1$.
\end{itemize}
We show that properties $8$ and $9$ are equivalent with the previous properties if $\PG(n,q)$ admits a $k$-spread.
\begin{itemize}
\item $2 \Rightarrow 8$: Let $\mathcal{S}$ be a $k$-spread in $\PG(n,q)$ and $\chi_\mathcal{S}^{}$ its characteristic vector. Then we know that $\chi_\mathcal{S}^{}-{\qbin nk}^{-1}\bm{j} \in \ker(A)$. 
Since $\chi \in \ker(A)^\perp$ we have that \[0=\chi\cdot \left(\chi_\mathcal{S}^{}-{\qbin nk}^{-1}\bm{j} \right) = |\mathcal{L}\cap \mathcal{S}|-|\mathcal{L}|{\qbin nk}^{-1},\] so $|\mathcal{L}\cap \mathcal{S}| = |\mathcal{L}|{\qbin nk}^{-1} =x$. 
\item $8 \Rightarrow 9$: Trivial. 
\item $9 \Rightarrow 3$: Suppose that $\PG(n,q)$ contains $k$-spreads, hence also regular $k$-spreads. We know that the group PGL$(n+1,q)$ acts transitively on the pairs of pairwise disjoint $k$-spaces. Let $n_i$, for $i=1,2$, be the number of regular $k$-spreads that contain $i$ fixed pairwise disjoint $k$-spaces. This number only depends on $i$, and not on the chosen $k$-spaces. \\
Let $\pi$ be a fixed $k$-space. The number of pairs $(\pi',\mathcal{S})$, with $\mathcal{S}$ a regular $k$-spread that contains $\pi$ and $\pi'$ is equal to $q^{(k+1)^2}\qbin{n-k}{k+1} \cdot n_2 = n_1 \cdot \left(\frac{q^{n+1}-1}{q^{k+1}-1}-1\right)$, so $n_1/n_2 = q^{k(k+1)}\qbin{n-k-1}{k}$.\\
By counting the number of pairs $(\pi',\mathcal{S})$, with $\pi'\in \mathcal{L}$ and $\mathcal{S}$ a regular $k$-spread that contains $\pi$ and $\pi'$, we find that the number of $k$-spaces in $\mathcal{L}$, disjoint from a fixed $k$-space $\pi$, is given by $(x-\chi(\pi))n_1/n_2 = (x-\chi(\pi))q^{k(k+1)}\qbin{n-k-1}{k}$.
\end{itemize}
To end this proof, we show that property $7$ is equivalent with the other properties.
\begin{itemize}
\item $2 \Rightarrow 7$: Let $\chi_\mathcal{R}^{}$ and $\chi_{\mathcal{R}'}^{}$ be the characteristic vectors of the pair of conjugate switching sets $\mathcal{R}$ and $\mathcal{R}'$ respectively. As $\mathcal{R}$ and $\mathcal{R}'$ cover the same set of points, we find: $\chi_\mathcal{R}^{}-\chi_{\mathcal{R}'}^{} \in \ker(A)$. This implies $0=\chi \cdot (\chi_\mathcal{R}^{}-\chi_{\mathcal{R}'}^{})=\chi \cdot \chi_\mathcal{R}^{}-\chi \cdot \chi_{\mathcal{R}'}^{}$, so that $\chi \cdot \chi_\mathcal{R}^{}=|\mathcal{L} \cap \mathcal{R}| = |\mathcal{L} \cap \mathcal{R}'|=\chi \cdot \chi_{\mathcal{R}'}^{}$.
\item $7 \Rightarrow 1$:  We first show that property $7$ implies the other properties if $n=2k+1$. 
For any two $k$-spreads $\mathcal{S}_1,\mathcal{S}_2$, 
the sets $\mathcal{S}_1 \setminus \mathcal{S}_2$ and $\mathcal{S}_2 \setminus \mathcal{S}_1 $ form a pair of conjugate switching sets. So $|\mathcal{L} \cap (\mathcal{S}_1 \setminus \mathcal{S}_2)|=|\mathcal{L}\cap (\mathcal{S}_2 \setminus \mathcal{S}_1 )|$, which implies that $|\mathcal{L} \cap \mathcal{S}_1|=|\mathcal{L}\cap \mathcal{S}_2|=c$.

Now we prove that this constant $c$ equals $x=|\mathcal{L}|\qbin{2k+1}{k}^{-1}$. Let $n_i$, for $i=0,1$, be the number of $k$-spreads containing $i$ fixed pairwise disjoint $k$-spaces. This number only depends on $i$, and not on the chosen $k$-spaces. The number of pairs $(\pi,\mathcal{S})$, with $\mathcal{S}$ a $k$-spread that contains $\pi$, is equal to $\qbin{2k+2}{k+1} \cdot n_1 = n_0 \cdot \frac{q^{2k+2}-1}{q^{k+1}-1}$, which implies that $n_0/n_1 = \qbin{2k+1}{k}$.

By counting the number of pairs $(\pi,\mathcal{S})$, with $\mathcal{S}$ a $k$-spread that contains $\pi$, and where $\pi\in \mathcal{L}$, we find, that the number of $k$-spaces in $\mathcal{L}\cap \mathcal{S}$ equals $|\mathcal{L}| n_1/n_0=|\mathcal{L}| \qbin{2k+1}{k} ^{-1}=x$. This implies property $8$, and hence, property $1$.

Now we prove that implication $7 \Rightarrow 1$ also holds if $n>2k+1$.
 Given a subspace $\tau$ in $\PG(n,q)$, we will use the notation $A_{|\tau}$ for the submatrix of $A$, where we only have the rows, corresponding with the points of $\tau$, and the columns corresponding with the $k$-spaces in $\tau$. We know that the matrix $A_{|\tau}$ has full rank by  Lemma \ref{lemmaAfullrow}.\\
Let $\Pi$ be a $(2k+1)$-dimensional subspace in $\PG(n,q)$. By property $7$, we know that for every two $k$-spreads $\mathcal{R},\mathcal{R}'$ in $\Pi$, we have $|\mathcal{L}\cap \mathcal{R}|=|\mathcal{L}\cap \mathcal{R}'|$ since $\mathcal{R}\setminus \mathcal{R}'$ and $\mathcal{R}'\setminus \mathcal{R}$ are conjugate switching sets. This implies that $\chi_{\mathcal{L}|\Pi}^{} \in \im\left(A_{|\Pi}^T\right)$ by the arguments above applied for the $(2k+1)$-space $\Pi$. So, there is a linear combination of the rows of $A_{|\Pi}$ equal to $\chi_{\mathcal{L}|\Pi}^{}$. This linear combination is unique since $A_{|\Pi}$ has full row rank.
Now we will show that the linear combination of $\chi_\mathcal{L}^{}$ is uniquely defined by the vectors $\chi_{\mathcal{L}|\Pi}^{}$, with $\Pi$ varying over all $(2k+1)$-spaces in $\PG(n,q)$.

We show, for every two $(2k+1)$-spaces $\Pi,\Pi'$, that the coefficients of the row corresponding to a point in $\Pi \cap \Pi'$ in the linear combination of $\chi_{\mathcal{L}|\Pi}^{}$ and in the linear combination of $\chi_{\mathcal{L}|\Pi'}^{}$ are equal.

Suppose $\chi_{\mathcal{L}|\Pi}^{} = a_1r_1+a_2r_2+\dots +a_lr_l+a_{l+1}r_{l+1}+\dots +a_mr_m$ and $\chi_{\mathcal{L}|\Pi'}^{} = b_{l+1}r_{l+1}+\dots +b_{m}r_{m}+b_{m+1}r_{m+1}+\dots +b_sr_s$, where $r_1, \dots, r_l, \dots ,r_m$ and $r_{l+1}, \dots, r_m, \dots ,r_s$ are the rows corresponding with the points of $\Pi$ and $\Pi'$, respectively. Remark that we only look at the columns corresponding with the $k$-spaces in $\Pi$ and $\Pi'$, respectively. 

We now look at the space $\Pi \cap \Pi'$, and at the corresponding columns in $A$.
Recall that $A_{|\Pi\cap\Pi'}$ also has full row rank, so the linear combination that gives $\chi_{\mathcal{L}|(\Pi\cap\Pi')}^{}$ is unique, and equal to the ones corresponding with $\Pi$ and $\Pi'$, restricted to $\Pi \cap \Pi'$. This proves that $a_i=b_i$ for $l+1 \leq i \leq m$. Here we also used the fact that the entry in $A$ corresponding with a point of $\Pi \setminus \Pi'$ or $\Pi' \setminus \Pi$ and a $k$-space in $\Pi \cap \Pi'$ is zero.

By using all $(2k+1)$-spaces, we see that $\chi_\mathcal{L}^{}$ is uniquely defined, and by construction we have $\chi_\mathcal{L}^{} \in \im(A^T)$. Note that we only used that property $7$ holds for conjugate switching sets inside a $(2k+1)$-dimensional subspace. \qedhere
\end{itemize}
\end{proof}
  
\begin{definition}
A set $\mathcal{L}$ of $k$-spaces in $\PG(n,q)$ that fulfills one of the statements in Theorem \ref{theodef} (and consequently all of them) is called a \emph{Cameron-Liebler set of $k$-spaces} in $\PG(n,q)$ with parameter $x=|\mathcal{L}|{\qbin nk}^{-1}$.
\end{definition}

\begin{remark}
Cameron-Liebler sets of $k$-spaces in $\PG(n,q)$ were introduced before in \cite{Ferdinand.} as we mentioned in the introduction. Remark that the definition we present here is consistent with the definition in \cite{Ferdinand.} since the definition given in that article is property 5.~from the previous theorem.
\end{remark}

Note that the parameter of a Cameron-Liebler set of $k$-spaces in $\PG(n,q)$ is not necessarily an integer, while the parameter of Cameron-Liebler line sets in $\PG(3,q)$ and the parameter of Cameron-Liebler sets of generators in polar spaces are integers (see \cite[Theorem 4.8]{CLpol}).

We end this section with showing an extra property of Cameron-Liebler sets of $k$-spaces in $\PG(n,q)$. 

\begin{lemma}
Let $\mathcal{L}$ be a Cameron-Liebler set of $k$-spaces in $\PG(n,q)$, then we find the following equality for every  $j$-dimensional subspace $\alpha$ and every $i$-dimensional subspace $\tau$, with $\alpha \subset \tau$ and $j< k<i$:
\begin{align*}
|[k]_\alpha \cap \mathcal{L}|+\frac{\qbin{n-j-1}{k-j}(q^{k-j}-1)}{\qbin{i}{k}(q^{i-k}-1)}|[k]^\tau \cap \mathcal{L}|=\frac{\qbin{n-j-1}{k-j}}{\qbin{i-j-1}{k-j}}|[k]^\tau_\alpha \cap \mathcal{L}|+\frac{\qbin{n-j-1}{k-j-1}}{\qbin nk}| \mathcal{L}|\;.
\end{align*}
Here $[k]_\alpha $, $[k]^\tau$ and $[k]^\tau_\alpha$ denote the set of all $k$-spaces through $\alpha$, the set of all $k$-spaces in $\tau$ and the set of all $k$-spaces in $\tau$ through $\alpha$, respectively.
\end{lemma}
\begin{proof}
Let $\chi_{[\alpha]}$, $\chi_{[\tau]}$ and $\chi_{[\alpha,\tau]}$ be the characteristic vectors of $[k]_\alpha$, $[k]^\tau$ and $[k]^\tau_\alpha$, respectively, and define
\[
v=\chi_{[\alpha]}+\frac{\qbin{n-j-1}{k-j}(q^{k-j}-1)}{\qbin{i}{k}(q^{i-k}-1)} \chi_{[\tau]}-\frac{\qbin{n-j-1}{k-j}}{\qbin{i-j-1}{k-j}} \chi_{[\alpha,\tau]} -\frac{\qbin{n-j-1}{k-j-1}}{\qbin nk}\bm{j}\;.
\]
By calculating $(Av)_{P'}$ for every point $P'$, we see that $Av=0$. This implies that $v\in \ker(A)$. Let $\chi$ be the characteristic vector of $\mathcal{L}$. By Definition 2 in Theorem \ref{theodef} we know that $\chi \in \ker(A)^\perp$, so by calculating $\chi \cdot v$ the lemma follows.
\end{proof}


For $k=1$, Drudge showed in \cite{phdDrudge} that this property is an equivalent definition for a Cameron-Liebler line set in $\PG(n,q)$. For $k>1$ we pose it as an open problem to show that this property is also an equivalent definition.

\section{Properties of Cameron-Liebler sets of \texorpdfstring{$k$}{k}-spaces in \texorpdfstring{$\PG(n,q)$}{PG(n,q)}}\label{sec3}

We start with some properties of Cameron-Liebler sets of $k$-spaces in $\PG(n,q)$ that can easily be proved.
\begin{lemma} \label{basislemma4}
Let $\mathcal{L}$ and $\mathcal{L}'$ be two Cameron-Liebler sets of $k$-spaces in $\PG(n,q)$ with parameters $x$ and $x'$ respectively, then the following statements are valid.
\begin{enumerate}
\item $0 \leq x \leq \qb{n+1}{k+1}$.
\item The set of all $k$-spaces in $\PG(n,q)$ not in $\mathcal{L}$ is a Cameron-Liebler set of $k$-spaces with parameter $\qb{n+1}{k+1}-x$.
\item If $\mathcal{L} \cap \mathcal{L}' = \emptyset$, then $\mathcal{L} \cup \mathcal{L}'$ is a Cameron-Liebler set of $k$-spaces with parameter $x+x'$.
\item  If $\mathcal{L}' \subseteq \mathcal{L}$, then $\mathcal{L} \setminus \mathcal{L}'$ is a Cameron-Liebler set of $k$-spaces with parameter $x-x'$.
\end{enumerate}
\end{lemma}

We present some examples of Cameron-Liebler $k$-sets in $\PG(n,q)$. 
\begin{example}\label{voorbeeldCL}
The set of all $k$-spaces through a point $P$ is a Cameron-Liebler set of $k$-spaces with parameter $1$ since the characteristic vector of this set is the row of $A$ corresponding to the point $P$. We will call this set of $k$-spaces the \emph{point-pencil through $P$}.\\
\end{example}

\begin{example}\label{voorbeeldCL2}
By property $3$ in Theorem \ref{theodef}, we can see that the set of all $k$-spaces in a fixed hyperplane is a Cameron-Liebler set of $k$-spaces in $\PG(n,q)$ with parameter $\frac{q^{n-k}-1}{q^{k+1}-1}$. 
Remark that this parameter is not an integer if $k+1 \nmid n+1$, or equivalently, if $\PG(n,q)$ does not contain a $k$-spread.
\end{example}

In \cite{Klaus} several properties of Cameron-Liebler sets of $k$-spaces in $\PG(2k+1,q)$ were given. We will first generalize some of these results to use them in Section \ref{nieuwhfdst}.

\begin{lemma} \label{driedisjunct}
Let $\pi$ and $\pi'$ be two disjoint $k$-spaces  in $\PG(n,q)$ with $\Sigma = \langle \pi,\pi' \rangle$, and let $P$ be a point in $\Sigma \setminus (\pi \cup \pi')$ and let $P'$ be a point not in $\Sigma$. Then the number of  $k$-spaces disjoint from $\pi$ and $\pi'$ equals $W(q,n,k)$, the number of $k$-spaces disjoint from $\pi$ and $\pi'$ through $P$ equals $W_\Sigma(q,n,k)$ and the number of $k$-spaces disjoint from $\pi$ and $\pi'$ through $P'$ equals $W_{\bar{\Sigma}}(q,n,k)$. 

Here, $W(q,n,k), W_{{\Sigma}}(q,n,k), W_{\bar{\Sigma}}(q,n,k)$ are given by:
\begin{align*}
W(q,n,k)&= \sum_{i=-1}^k  W_i(q,n,k) \\
W_\Sigma(q,n,k) &= \frac{1}{(q^{k+1}-1)^2} \sum_{i=0}^k W_i(q,n,k) (q^{i+1}-1)
\\
W_{\bar{\Sigma}}(q,n,k) &= \frac{1}{q^{n+1}-q^{2k+2}}  \sum_{i=-1}^{k-1} W_i(q,n,k) (q^{k+1}-q^{i+1})  \\
W_i(q,n,k)&=
\begin{cases}
q^{2k^2+k+ \frac{3i^2}{2}-\frac{i}{2}-3ik}\qbin{n-2k-1}{k-i}\qbin{k+1}{i+1}\prod_{j=0}^i   (q^{k-j+1}-1)  & \text{if } i \geq 0\\
q^{2(k+1)^2}\qbin{n-2k-1}{k+1} & \text{if } i=-1
\end{cases}\;.
\end{align*}
\end{lemma}
\begin{proof}
 To count the number of $k$-spaces $\pi''$, that are disjoint from $\pi$ and $\pi'$, we first count the number of possible intersections $\pi'' \cap \Sigma$.

We count the number of $i$-spaces in $\Sigma$, disjoint from $\pi$ and $\pi'$, by counting $((P_0,P_1,\dots, P_i),\sigma_i)$ in two ways. Here $\sigma_i$ is an $i$-space in $\Sigma$, disjoint from $\pi$ and $\pi'$, and the points $P_0,P_1, \dots, P_i$ form a basis of $\sigma_i$. For the ordered basis $(P_0,P_1, \dots, P_i)$ we have $\prod_{j=0}^{i} \frac{q^{2j}(q^{k-j+1}-1)^2}{q-1}$ possibilities since there are $\qbin{2k+2}{1}-2\qbin{k+j+1}{1}+\qbin{2j}{1}=\frac{q^{2j}(q^{k-j+1}-1)^2}{q-1}$ possibilities for $P_j$ if $P_0,P_1,\dots,P_{j-1}$ are given. 

By a similar argument, we find that the number of ordered bases $(P_0,P_1, \dots, P_i)$ for a given $\sigma_i$ is $\prod_{j=0}^{i} \frac{q^{j}(q^{i-j+1}-1)}{q-1}$.

In this way we find that the number of $i$-spaces in $\Sigma$, disjoint from $\pi$ and $\pi'$, is given by:
\begin{align*}
\frac{\prod_{j=0}^{i} \frac{q^{2j}(q^{k-j+1}-1)^2}{q-1}}{\prod_{j=0}^{i} \frac{q^{j}(q^{i-j+1}-1)}{q-1}}
=\prod_{j=0}^{i}\frac{ q^{j}(q^{k-j+1}-1)^2}{q^{i-j+1}-1}= q^{\binom{i+1}{2}}\qbin{k+1}{i+1}\prod_{j=0}^i(q^{k-j+1}-1).
\end{align*}
Now we count, for a given $i$-space $\sigma_i$ in $\Sigma$, the number of $k$-spaces $\pi''$ through $\sigma_i$ such that $\pi'' \cap \Sigma = \sigma_i$. This equals the number of $(k-i-1)$-spaces in $\PG(n-i-1,q)$, disjoint from a $(2k-i)$-space. This number is $q^{(k-i)(2k-i+1)}\qbin{n-2k-1}{k-i}$ by Lemma \ref{lemmadisjunct}. By this lemma we also see that the number of $k$-spaces disjoint from $\Sigma$ is given by $q^{(k+1)(2k+2)}\qbin{n-2k-1}{k+1}$. 
This implies that $W_i(q,n,k), -1 \leq i\leq k$, is the  number of $k$-spaces disjoint from $\pi$ and $\pi'$, and intersecting $\Sigma$ in an $i$-space. 

Now we have enough information to count the number of $k$-spaces disjoint from $\pi$ and $\pi'$:
\begin{align*}
W(q,n,k)&=\sum_{i=-1}^k  W_i(q,n,k)\;.
\end{align*}
We use the same arguments to calculate $W_\Sigma(q,n,k)$ and $W_{\bar{\Sigma}}(q,n,k)$. By double counting $(P, \pi'')$, with $\pi''$ a $k$-space through $P\in \Sigma$ disjoint from $\pi$ and $\pi'$, and double counting  $(P', \pi'')$, with $\pi''$ a $k$-space through $P'\notin \Sigma$ disjoint from $\pi$ and $\pi'$,  we find:
\begin{align*}
\left( \qbin{2k+2}{1} -2\qbin{k+1}{1}\right) \cdot W_\Sigma(q,n,k) &= \sum_{i=0}^k W_i(q,n,k) \cdot \qbin{i+1}{1}  \text{ and}  \\
\left( \qbin{n+1}{1} -\qbin{2k+2}{1}\right) \cdot W_{\bar{\Sigma}}(q,n,k) &= \sum_{i=-1}^{k-1} W_i(q,n,k) \cdot \left(\qbin{k+1}{1}-\qbin{i+1}{1}  \right)\;.
\end{align*}
This implies:
\begin{align*}
W_\Sigma(q,n,k) =& \frac{1}{(q^{k+1}-1)^2} \sum_{i=0}^k W_i(q,n,k) (q^{i+1}-1)
\\
W_{\bar{\Sigma}}(q,n,k)  =& \frac{1}{q^{n+1}-q^{2k+2}}  \sum_{i=-1}^{k-1} W_i(q,n,k)(q^{k+1}-q^{i+1})\;. \qedhere
\end{align*} 
\end{proof}

From now on we denote $W_i(q,n,k), W_\Sigma (q,n,k)$ and $W_{\bar{\Sigma}}(q,n,k)$ by $W_i, W_\Sigma$ and $W_{\bar{\Sigma}}$ if the dimensions $n$, $k$ and the field size $q$ are clear from the context. 
\begin{lemma} \label{lemmas1s2d1d2}
Let $\mathcal{L}$ be a Cameron-Liebler set of $k$-spaces in $\PG(n,q)$ with parameter $x$. 
\begin{enumerate}
\item For every $\pi \in \mathcal{L}$, there are $s_1$ elements of $\mathcal{L}$ meeting $\pi$. 
\item For skew $\pi, \pi'\in \mathcal{L}$ and a $k$-spread $\mathcal{S}_0$ in $\Sigma = \langle \pi,\pi' \rangle$, there exist exactly $d_2$ subspaces in $\mathcal{L}$ that are skew to both $\pi$ and $\pi'$ and there exist $s_2$ subspaces in $\mathcal{L}$ that meet both $\pi$ and $\pi'$.

Here, $d_2$, $s_1$ and $s_2$ are given by:
\begin{align*}
d_2(q,n,k,x,\mathcal{S}_0) &= (W_\Sigma-W_{\bar{\Sigma}})|\mathcal{S}_0 \cap \mathcal{L}|-2W_\Sigma+x W_{\bar{\Sigma}}\\
s_1(q,n,k,x) &= x\qbin{n}{k}-(x-1)\qbin{n-k-1}{k}q^{k^2+k}\\
s_2(q,n,k,x,\mathcal{S}_0) &= x\qbin{n}{k}-2(x-1)\qbin{n-k-1}{k}q^{k^2+k} +d_2(q,n,k,x,\mathcal{S}_0)\;,\\
\end{align*}
where $W_\Sigma$ and $W_{\bar{\Sigma}}$ are given by Lemma \ref{driedisjunct}.
\item Define $d'_2(q,n,k,x) = (x-2)W_\Sigma$ and $s'_2(q,n,k,x) = x\qbin nk -2(x-1)\qbin{n-k-1}{k} q^{k^2+k} +d'_2(q,n,k,x)$. If $n>3k+1$, then $|\mathcal{S}_0\cap \mathcal{L}|\leq x$ for every $k$-spread $\mathcal{S}_0$ in $\Sigma$. Moreover we have that $d_2(q,n,k,x,\mathcal{S}_0) \leq d'_2(q,n,k,x)$ and $s_2(q,n,k,x,\mathcal{S}_0) \leq s'_2(q,n,k,x)$.
\end{enumerate}
\end{lemma}
\begin{proof}
\begin{enumerate}
\item This follows directly from Theorem \ref{theodef}$(3)$ and $|\mathcal{L}|=x\qbin nk$. 
\item Let $\chi_\pi$ and $\chi_{\pi'}$ be the characteristic vectors of $\{\pi\}$ and $\{\pi'\}$, respectively, and let $\mathcal{Z}$ be the set of all $k$-spaces in $\PG(n,q)$ disjoint from $\pi$ and $\pi'$, and let $\chi_\mathcal{Z}$ be its characteristic vector. Furthermore, let $v_\pi$ and $v_{\pi'}$ be the incidence vectors of $\pi$ and $\pi'$, respectively, with their positions corresponding to the points of $\PG(n,q)$. Note that $A\chi_\pi = v_\pi$ and $A\chi_{\pi'} = v_{\pi'}$. By Lemma \ref{driedisjunct} we know the numbers $W_\Sigma$ and $W_{\bar{\Sigma}}$ of $k$-spaces disjoint from $\pi$ and $\pi'$, through a point $P$, if $P\in \Sigma$ and $P \notin \Sigma$ respectively. Let $\mathcal{S}_0$ be a $k$-spread in $\Sigma$ and let $v_\Sigma$ be the incidence vector of $\Sigma$ (as a point set). We find: 
\begin{align*}
A\chi_\mathcal{Z} &=W_\Sigma(v_\Sigma-v_\pi-v_{\pi'})+W_{\bar{\Sigma}} (\bm{j}-v_{\Sigma} ) \\
&=W_\Sigma(A\chi_{\mathcal{S}_0}^{}-A\chi_\pi-A\chi_{\pi'})+W_{\bar{\Sigma}} \left(\qbin{n}{k}^{-1}A\bm{j}-A \chi_{\mathcal{S}_0}\right)\\
\Leftrightarrow\qquad&\chi_\mathcal{Z}-W_\Sigma(\chi_{\mathcal{S}_0}-\chi_\pi-\chi_{\pi'})-W_{\bar{\Sigma}} \left(\qbin{n}{k}^{-1}\bm{j}- \chi_{\mathcal{S}_0}\right) \in \ker(A).
\end{align*}
We know that the characteristic vector $\chi$ of $\mathcal{L}$ is included in $\ker(A)^\perp$. This implies:
\begin{align*}
&&\chi_\mathcal{Z} \cdot \chi &=W_\Sigma(\chi_{\mathcal{S}_0}\cdot\chi-\chi(\pi)-\chi(\pi'))+W_{\bar{\Sigma}} (x- \chi_{\mathcal{S}_0}\cdot\chi) \\
&\Leftrightarrow & |\mathcal{Z}\cap \mathcal{L}|  &=W_\Sigma(|\mathcal{S}_0\cap \mathcal{L}|-2)+W_{\bar{\Sigma}} (x- |\mathcal{S}_0\cap \mathcal{L}|)  \\
&\Leftrightarrow & |\mathcal{Z}\cap \mathcal{L}| &=(W_\Sigma-W_{\bar{\Sigma}})|\mathcal{S}_0\cap \mathcal{L}|-2W_\Sigma+x W_{\bar{\Sigma}}\;, 
\end{align*}
which gives the formula for $d_2(q,n,k,x)$.  The formula for $s_2(q,n,k,x)$ follows from the inclusion-exclusion principle.  
\item Suppose $\Sigma$ is a $(2k+1)$-space in $\PG(n,q)$, and suppose $\mathcal{S}_0$ is a $k$-spread in $\Sigma$ such that $|\mathcal{S}_0 \cap \mathcal{L}|> x$. By property $1$ in Theorem \ref{theodef} we know that the characteristic vector $\chi$ of $\mathcal{L}$ can be written as $\sum_{P \in \PG(n,q)} x_P r_P^T$ for some $x_{P}\in\R$ where $r_{P}$ is the row of $A$ corresponding to the point $P$. Let $\chi_{\pi}$ be the characteristic vector of the set $\{\pi\}$ with $\pi$ a $k$-space, then $\chi_{\pi} \cdot \chi=\sum_{P \in \pi} x_P$ equals $1$ if $\pi \in \mathcal{L}$ and $0$ if $\pi \notin \mathcal{L}$. As $\chi \cdot \bm{j} = |\mathcal{L}| = x\qbin nk$ we find that $\sum_{P \in \PG(n,q)} x_P = x$.
\\
If $|\mathcal{S}_0 \cap \mathcal{L}|> x$, then $\chi \cdot \chi_{S_0} = \sum_{P\in \Sigma}x_P >x$. From these observations, it follows that $\sum_{P \in \PG(n,q) \setminus \Sigma} x_P = \sum_{P \in \PG(n,q) } x_P - \sum_{P \in  \Sigma} x_P$  is negative. As $n>3k+1$, there exists a $k$-space $\tau$ in $\PG(n,q)$, disjoint from $\Sigma$, with $\chi_{\tau} \cdot \chi = \sum_{P\in \tau} x_P$ negative, which gives the contradiction. 

It follows that $|\mathcal{S}_0 \cap \mathcal{L}|\leq x$. Since this is true for every $k$-spread $\mathcal{S}_0$ in every $(2k+1)$-space in $\PG(n,q)$, the statement holds.
\qedhere
\end{enumerate}
\end{proof}
Remark that we will use the upper bound $d'_2(q,n,k,x)$ and $s'_2(q,n,k,x)$ instead of $d_2(q,n,k,x,\mathcal{S}_0)$ and $s_2(q,n,k,x,\mathcal{S}_0)$ respectively, since they are independent of the chosen $k$-spread $\mathcal{S}_0$.

The following lemma is a generalization of Lemma $2.4$ in \cite{Klaus}.
\begin{lemma}\label{lemmaklaus}
Let $c,n,k$ be nonnegative integers with $n>3k+1$ and 
\begin{align*}
(c+1)s_1-\binom{c+1}{2}s'_2 > x\qbin{n}{k}\;,
\end{align*}
then no Cameron-Liebler set of $k$-spaces in $\PG(n,q)$ with parameter $x$ contains $c+1$ mutually skew $k$-spaces. 
\end{lemma}
\begin{proof}
Assume that $\PG(n,q)$ has a Cameron-Liebler set $\mathcal{L}$ of $k$-spaces with parameter $x$ that contains $c+1$ mutually disjoint $k$-spaces $\pi_0,\pi_1,\dots,\pi_c$. Lemma \ref{lemmas1s2d1d2} shows that $\pi_i$ meets at least $s_1(q,n,k,x)-i s_2(q,n,k,x)$ elements of $\mathcal{L}$ that are skew to $\pi_0, \pi_1, \dots,\pi_{i-1}$. This implies that $x\qbin nk = |\mathcal{L}| \geq (c+1) s_1-\sum_{i=0}^c i s_2 \geq (c+1) s_1-\sum_{i=0}^c i s'_2$ which contradicts the assumption.
\end{proof}

\section{Classification results}\label{nieuwhfdst}
In this section, we will list some classification results for Cameron-Liebler sets of $k$-spaces in $\PG(n,q)$. First note that a Cameron-Liebler set of $k$-spaces with parameter $0$ is the empty set.\\
In the following lemma we start with the classification for the parameters $x  \in \ ]0,1[ \ \cup \ ]1,2[$.
\begin{lemma} \label{lemmatussen012}
There are no Cameron-Liebler sets of $k$-spaces in $\PG(n,q)$ with parameter $x \in \  ]0,1[ $ and if $n\geq3k+2$, then there are no Cameron-Liebler sets of $k$-spaces with parameter $x \in \ ]1,2[ $.
\end{lemma}
\begin{proof}
Suppose there is a Cameron-Liebler set $\mathcal{L}$ of $k$-spaces with parameter $x\in \ ]0,1[$. Then $\mathcal{L}$ is not the empty set so suppose $\pi \in \mathcal{L}$. By property $3$ in Theorem $\ref{theodef}$ we find that the number of $k$-spaces in $\mathcal{L}$ disjoint from $\pi$ is negative, which gives the contradiction. 

Suppose there is a Cameron-Liebler set $\mathcal{L}$ of $k$-spaces with parameter $x \in \ ]1,2[$ in $\PG(n,q)$, $n\geq 3k+2$. By property $3$ in Theorem \ref{theodef}, we know that there are at least two disjoint $k$-spaces $\pi,\pi'\in \mathcal{L}$.  By Lemma \ref{lemmas1s2d1d2}$(2,3)$ we know that there are $d_2 \leq d_2'$ elements of $\mathcal{L}$ disjoint from $\pi$ and $\pi'$. Since $d_2'$ is negative for $x\in ]1,2[$, we find a contradiction.
\end{proof}


We continue with a classification result for Cameron-Liebler $k$-sets with parameter $x=1$, where we will use the following result, the so-called Erd\H{o}s-Ko-Rado theorem for projective spaces.

\begin{theorem}[{\cite{EKR1,EKR2} }]\label{EKR1}
If $\mathcal{L}$ is a set of pairwise non-trivially intersecting $k$-spaces in $\PG(n,q)$ with $n\geq2k+1$, then $|\mathcal{L}| \leq \qbin nk$, and equality holds if and only if $\mathcal{L}$ either consists of all $k$-spaces through a fixed point, or $n = 2k+1$ and $\mathcal{L}$ consists of all
$k$-spaces in a fixed hyperplane.
\end{theorem}

\begin{theorem}\label{xgelijkaanee}
Let $\mathcal{L}$ be a Cameron-Liebler set of $k$-spaces with parameter $x=1$ in $\PG(n,q)$, $n\geq2k+1$. Then $\mathcal{L}$ is a point-pencil or $n=2k+1$ and $\mathcal{L}$ is the set of all $k$-spaces in a hyperplane of $\PG(2k+1,q)$.
\end{theorem}
\begin{proof}
The theorem follows immediately from Lemma \ref{EKR1} since, by Theorem \ref{theodef}$(3)$, we know that $\mathcal{L}$ is a family of pairwise intersecting $k$-spaces of size $\qbin nk$. 
\end{proof}

We continue this section by showing that there are no Cameron-Liebler sets of $k$-spaces in $\PG(n,q)$, $n\geq3k+2$, with parameter $2\leq x\leq \frac{1}{\sqrt[8]{2}} q^{\frac{n}{2}-\frac{k^2}{4}-\frac{3k}{4}-\frac{3}{2}}(q-1)^{\frac{k^2}{4}-\frac{k}{4}+\frac{1}{2}} \sqrt{q^2+q+1}$. For this classification result, we will use the following theorem, the so called Hilton-Milner theorem for projective spaces.

\begin{theorem}[{\cite[Theorem 1.4]{Mussche} }]\label{theomussche}
Let $k\geq 1$ be an integer. If $q\geq3$ and $n\geq 2k+2$, or if $q=2$ and $n\geq 2k+3$, then any family $\mathcal{F}$ of pairwise non-trivially intersecting $k$-spaces of $\PG(n,q)$, with $\cap_{F \in \mathcal{F}} F = \emptyset$ has size at most $\qbin nk -q^{k^2+k}\qbin{n-k-1}{k} + q^{k+1}$.
\end{theorem}

To simplify the notations, we denote $q^{\frac{n}{2}-\frac{k^2}{4}-\frac{3k}{4}-\frac{3}{2}}(q-1)^{\frac{k^2}{4}-\frac{k}{4}+\frac{1}{2}} \sqrt{q^2+q+1}$ by $f(q,n,k)$.\\
Recall that the set of all {$k$-spaces} in a hyperplane in $\PG(n,q)$ is a Cameron-Liebler set of $k$-spaces with parameter $x=\frac{q^{n-k}-1}{q^{k+1}-1}$ (see Example \ref{voorbeeldCL2}) and note that $f(q,n,k) \in \mathcal{O}(\sqrt{q^{n-2k}})$ while $\frac{q^{n-k}-1}{q^{k+1}-1} \in \mathcal{O}(q^{n-2k-1})$.

We start with some lemmas.
\begin{lemma}\label{ongelijkheid} For $n\geq 2k+2$, we have
\begin{align*}
	&\qbin{n}{k}> \qbin{n-k-1}{k}q^{k^2+k}>W_\Sigma\;.
\end{align*}
If also $k\geq2$, then
\begin{align*}
	\qbin{n-k-1}{k}q^{k^2+k}>q^{nk-k^2} + q^{nk-k^2-1} +q^{nk-k^2-2}\;.
\end{align*}
\end{lemma} 
\begin{proof}
The first inequality follows since $\qbin{n}{k}$ is the number of $k$-spaces through a fixed point in $\PG(n,q)$, $\qbin{n-k-1}{k}q^{k^2+k}$ is the number of $k$-spaces through a fixed point disjoint from a given $k$-space not through that point (see Lemma \ref{lemmadisjunct}), and $W_\Sigma$ is the number of $k$-spaces through a fixed point and disjoint from two given $k$-spaces not through that point.

The second inequality, for $k\geq 2, n\geq 2k+2$, follows from
\begin{align*}
\qbin{n-k-1}{k}q^{k^2+k}&=\left(\prod_{i=0}^{k-3} \left( \frac{q^{n-k-1-i}-1}{q^{k-i}-1}\right) \right)\left(\frac{q^{n-2k+1}-1}{q^{}-1}\frac{q^{n-2k}-1}{q^2-1} \right) q^{k^2+k}\\
&> q^{(n-2k-1)(k-2)} (q^{n-2k}+q^{n-2k-1}+q^{n-2k-2})q^{n-2k-2} q^{k^2+k}\\
&= q^{nk-k^2} + q^{nk-k^2-1} +q^{nk-k^2-2}\;.\qedhere
\end{align*}
\end{proof}

\begin{lemma}[{\cite[Lemma 2.1 and 2.2]{ferdinand} }]\label{ferdinand}
	Let $n>k>0$ be integers.
	\begin{itemize}
		\item If $q= 2$, then $\left(2-\frac{1}{q^{n-k}} \right)q^{k(n-k)} \leq \qbin{n}{k}\leq \frac{111}{32}q^{k(n-k)}$.
		\item If $q\geq 3$, then $\left(1+\frac{1}{q} \right)q^{k(n-k)} \leq \qbin{n}{k}\leq 2q^{k(n-k)}$.
	\end{itemize}
\end{lemma}

\begin{notation}\label{deltaenC}
	We denote $\Delta(q,n,k) = \qbin{n-k-1}{k}q^{k^2+k}$ and $C(q,n,k)=\qbin nk - \qbin{n-k-1}{k}q^{k^2+k}$ from now on. Then, according to Lemma \ref{lemmas1s2d1d2} we can write
	\begin{align*}
		s_1(q,n,k,x)&=xC(q,n,k)+\Delta(q,n,k)\quad\text{and}\\
		s_2'(q,n,k,x)&=xC(q,n,k)+(2-x)\Delta(q,n,k)+(x-2)W_\Sigma.
	\end{align*}
	We denote $\Delta(q,n,k)$ and $C(q,n,k)$ by $\Delta$ and $C$ if $q,n$ and $k$ are clear from the context.
\end{notation}

Before proving a lemma on $\Delta$ and $C$, we give a result on the Gaussian binomial coefficients. First, we recall the (double) $q$-analogue of Pascal's rule:
\begin{align}\label{qpascal}
	q^{b}\qbin{a-1}{b}_{q}+\qbin{a-1}{b-1}_{q}=\qbin{a}{b}_{q}=\qbin{a-1}{b}_{q}+q^{a-b}\qbin{a-1}{b-1}_{q}\;.
\end{align}

\begin{lemma}\label{somproductgbc}
	For integers $a,b,c$ with $0\leq b,c\leq a$ we have that
	\begin{align*}
		\qbin{a}{b}_{q}=\sum_{i=0}^{c}\qbin{a-c}{b-i}_{q}\qbin{c}{i}_{q}q^{(b-i)(c-i)\;.}
	\end{align*}
	Here, we consider $\qbin{x}{y}=0$ if $y<0$ or $y>x$.
\end{lemma}
\begin{proof}
	Induction on $c$. In the induction step we use the left equality in \eqref{qpascal}.
\end{proof}

\begin{lemma}\label{LemmaWsigma} If $n\geq 2k+1$ and $q\geq3$, then 
	\begin{align*}
		W_\Sigma \leq  \Delta-\frac{C}{2}.
	\end{align*}
\end{lemma}
\begin{proof}
	First, using the definition of $W_\Sigma$ as given in Lemma \ref{driedisjunct}, we find
	\begin{align*}
		W_\Sigma&=\frac{1}{(q^{k+1}-1)^2} \sum_{i=0}^{k}(q^{i+1}-1)q^{2k^2+k+ \frac{3i^2}{2}-\frac{i}{2}-3ik}\qbin{n-2k-1}{k-i}\qbin{k+1}{i+1}\prod_{j=0}^{i}(q^{k-j+1}-1)\\
		&=q^{2k^2+k}\sum_{i=0}^{k}q^{ \frac{3i^2}{2}-\frac{i}{2}-3ik}\qbin{n-2k-1}{k-i}\qbin{k}{i}\prod_{j=1}^{i}(q^{k-j+1}-1)\;.
	\end{align*}
	Here, the final product is considered 1 if $i=0$ (the `empty' product). Now, using the definitions of $\Delta$ and $C$ as in Notation \ref{deltaenC}, the inequality stated above can be written as:
	\begin{align}\label{ongelijkheid1}
		q^{2k^2+k}\sum_{i=0}^{k}q^{ \frac{3i^2}{2}-\frac{i}{2}-3ik}\qbin{n-2k-1}{k-i}\qbin{k}{i}\prod_{j=1}^{i}(q^{k-j+1}-1)\leq \frac{3}{2}\qbin{n-k-1}{k}q^{k^2+k}-\frac{1}{2}\qbin{n}{k}\;.
	\end{align}
	For $k=1$ it reduces to
	\begin{align*}
		q^{3}\qbin{n-3}{1}+q(q-1)\leq \frac{3}{2}\qbin{n-2}{1}q^{2}-\frac{1}{2}\qbin{n}{1}\qquad	\Leftrightarrow\qquad\frac{q-1}{2}\geq 0\;,
	\end{align*}
	which is true for all $q\geq2$. So, we will from now on assume that $k\geq2$.
	\par Repeatedly applying the left equality in \eqref{qpascal} we find that $\qbin{n}{k}=q^{k^{2}+k}\qbin{n-k-1}{k}+\sum_{i=0}^{k}q^{ik}\qbin{n-i-1}{k-1}$, so inequality \eqref{ongelijkheid1} can be rewritten as
	\begin{multline*}
		q^{2k^2+k}\sum_{i=0}^{k}q^{ \frac{3i^2}{2}-\frac{i}{2}-3ik}\qbin{n-2k-1}{k-i}\qbin{k}{i}\prod_{j=1}^{i}(q^{k-j+1}-1)+\frac{1}{2}\sum_{i=0}^{k}q^{ik}\qbin{n-i-1}{k-1}\\\leq \qbin{n-k-1}{k}q^{k^2+k}\;.
	\end{multline*}
	We now apply Lemma \ref{somproductgbc} on the right hand side of this inequality and we see that it is equivalent with
	\begin{align}\label{ongelijkheid2}
		&q^{2k^2+k}\sum_{i=0}^{k}q^{ \frac{3i^2}{2}-\frac{i}{2}-3ik}\qbin{n-2k-1}{k-i}\qbin{k}{i}\prod_{j=1}^{i}(q^{k-j+1}-1)+\frac{1}{2}\sum_{i=0}^{k}q^{ik}\qbin{n-i-1}{k-1}\nonumber\\&\qquad\qquad\qquad\leq q^{k^2+k}\sum_{i=0}^{k}q^{ (k-i)^{2}}\qbin{n-2k-1}{k-i}\qbin{k}{i}\nonumber\\
		\Leftrightarrow\quad&q^{2k^2+k}\sum_{i=1}^{k}q^{ \frac{3i^2}{2}-\frac{i}{2}-3ik}\qbin{n-2k-1}{k-i}\qbin{k}{i}\prod_{j=1}^{i}(q^{k-j+1}-1)+\frac{1}{2}\sum_{i=0}^{k}q^{ik}\qbin{n-i-1}{k-1}\nonumber\\&\qquad\qquad\qquad\leq q^{k^2+k}\sum_{i=1}^{k}q^{(k-i)^{2}}\qbin{n-2k-1}{k-i}\qbin{k}{i}\;.
	\end{align}
	Now, we note that $\prod_{j=1}^{i}(q^{k-j+1}-1)\leq q^{(i-1)(k+1)-\frac{i(i-1)}{2}}(q^{k-i+1}-1)$ for $i\geq1$. So, in order to prove \eqref{ongelijkheid2}, it is sufficient to show that the following inequality is valid:
	\begin{align}\label{ongelijkheid3}
		\frac{1}{2}\sum_{i=0}^{k}q^{ik}\qbin{n-i-1}{k-1}&\leq q^{k^2+k}\sum_{i=1}^{k}\left(q^{(k-i)^{2}}-q^{(k-i)(k-i-1)-1}(q^{k-i+1}-1)\right)\qbin{n-2k-1}{k-i}\qbin{k}{i}\nonumber\\
		&=q^{k^2+k}\sum_{i=1}^{k}q^{(k-i)(k-i-1)-1}\qbin{n-2k-1}{k-i}\qbin{k}{i}\nonumber\\
		&=q^{2k^{2}-2k+1}\qbin{n-2k-1}{k-1}\qbin{k}{1}+q^{k^2+k}\sum_{i=2}^{k}q^{(k-i)(k-i-1)-1}\qbin{n-2k-1}{k-i}\qbin{k}{i}\;.
	\end{align}
	Applying Lemma \ref{ferdinand} on the left hand side in \eqref{ongelijkheid3} we find that
	\begin{align}\label{links}
		\frac{1}{2}\sum_{i=0}^{k}q^{ik}\qbin{n-i-1}{k-1}\leq q^{(k-1)(n-k)}\sum_{i=0}^{k}q^{i}=q^{(k-1)(n-k)}\frac{q^{k+1}-1}{q-1}\;.
	\end{align}
	Now applying Lemma \ref{ferdinand} on the first term of the right hand side in \eqref{ongelijkheid3} we find that
	\begin{align}\label{rechts}
		q^{2k^{2}-2k+1}\qbin{n-2k-1}{k-1}\qbin{k}{1}\geq\left(1+\frac{1}{q}\right)q^{(k-1)(n-k)+1}\frac{q^{k}-1}{q-1}=(q+1)q^{(k-1)(n-k)}\frac{q^{k}-1}{q-1}\;.
	\end{align}
	From \eqref{links} and \eqref{rechts} it follows that in order to prove \eqref{ongelijkheid3}, it is sufficient to show that the following inequality is valid:
	\begin{align*}
		q^{k+1}-1\leq(q+1)\left(q^{k}-1\right) \quad\Leftrightarrow\quad q^{k}\geq q\;,
	\end{align*}
	This statement is clearly true.
\end{proof}

\begin{lemma}\label{lemmadriehoek}
	If $x\leq \frac{1}{\sqrt[8]{2}}f(q,n,k)$ and $n\geq 2k+2$, 
	then $\frac{\Delta}{C}>\sqrt[4]{2}x^2$.
\end{lemma}
\begin{proof}
	We want to prove that 
	\begin{align*}
	 \qbin{n-k-1}{k} q^{k^2+k} >\sqrt[4]{2} x^2 \left( \qbin nk -\qbin{n-k-1}{k} q^{k^2+k} \right).
	\end{align*}
	We first look at the case $k\geq2$. Given a $k$-space $\pi$ in $\PG(n-1,q)$, the number of $(k-1)$-spaces meeting $\pi$ equals $\qbin nk -\qbin{n-k-1}{k} q^{k^2+k}$ by Lemma \ref{lemmadisjunct}. We know that this number is smaller than the product of the number of points $Q \in \pi$ and the number of $(k-1)$-spaces through $Q$. This implies that
	\begin{align*}
	\qbin{n}{k} -\qbin{n-k-1}{k} q^{k^2+k} &\leq \qbin{k+1}{1}\qbin{n-1}{k-1} \\
	&= \frac{q^{k+1}-1}{q-1}\cdot\qbreuk{n-1}{n-k+1}{k-1}{} \\
	&< \frac{q^{nk-\frac{k^2}{2}-n+\frac{3k}{2}+1}}{(q-1)^{\frac{k^2}{2}-\frac{k}{2}+1}}\;.
	\end{align*}
	From this computation and the assumption on $x$ it follows that
	\begin{align*}
		\sqrt[4]{2} x^2 \left( \qbin nk -\qbin{n-k-1}{k} q^{k^2+k} \right)<(f(q,n,k))^2 \frac{q^{nk-\frac{k^2}{2}-n+\frac{3k}{2}+1}}{(q-1)^{\frac{k^2}{2}-\frac{k}{2}+1}}&=q^{nk-k^2-2}(q^2+q+1)\\
		&\leq\qbin{n-k-1}{k} q^{k^2+k}\;,
	\end{align*}
	where the final inequality is given by Lemma \ref{ongelijkheid} (which we can apply since $k\geq2$).
	\par Now we look at the case $k=1$. We have to prove that 
	\begin{align*}
		\qbin{n-2}{1} q^{2} >\sqrt[4]{2} x^2 \left( \qbin{n}{1} -\qbin{n-2}{1} q^{2} \right)
		\quad\Leftrightarrow\quad \frac{q^{n-2}-1}{q^2-1}q^2> \sqrt[4]{2} x^2\;.
	\end{align*}
	By the assumption on $x$ it is sufficient to prove that
	\begin{align*}
		\frac{q^{n-2}-1}{q^2-1}q^2> f(q,n,1)^{2}=q^{n-5}(q^3-1)\quad\Leftrightarrow\quad q^{n-2}+q^{n-3}-q^{n-5}-q^{2}>0\;,
	\end{align*}
	which is clearly true since $n\geq 4$.
\end{proof}

\begin{lemma}\label{lemmaongelijkheidklauspargroterdanmeer}
	Let $\mathcal{L}$ be a Cameron-Liebler set of $k$-spaces in $\PG(n,q)$, $n\geq3k+2$, with parameter $2\leq  x\leq \frac{1}{\sqrt[8]{2}} f(q,n,k)$, then $\mathcal{L}$ cannot contain $\left \lfloor \frac{3}{2}x\right \rfloor$ mutually disjoint $k$-spaces.
\end{lemma}
\begin{proof}
	We apply Lemma \ref{lemmaklaus}, with $c+1=\left \lfloor \frac{3}{2}x\right \rfloor$ and have to show that
	\begin{align*}
		\left \lfloor \frac{3}{2}x\right \rfloor s_1 - \binom{\left \lfloor \frac{3}{2}x\right \rfloor}{2}s'_2 > x\qbin nk\;.
	\end{align*}
	Using Notation \ref{deltaenC} and Lemma \ref{LemmaWsigma} we see that it is sufficient to prove that
	\begin{align*} 
		&\left \lfloor \frac{3}{2}x\right \rfloor \left(xC+\Delta\right)-x(\Delta +C)\\&\qquad -\frac{1}{2}\left \lfloor \frac{3}{2}x\right \rfloor\left(\left\lfloor \frac{3}{2} x\right\rfloor-1\right)\left(xC-(x-2)\Delta+(x-2)\left(\Delta-\frac{C}{2}\right)\right) > 0\\
		\Leftrightarrow\qquad& \Delta \left( \left \lfloor \frac{3}{2}x\right \rfloor-x\right)> C\left(x-\left \lfloor \frac{3}{2}x\right \rfloor x+\frac{1}{2}\left \lfloor \frac{3}{2}x\right \rfloor\left(\left \lfloor \frac{3}{2}x\right \rfloor-1\right)\left(\frac{x}{2}+1\right) \right)\;.
	\end{align*}
	From Lemma \ref{lemmadriehoek}, we know that $\frac{\Delta}{C}>\sqrt[4]{2}x^2$. Hence it is sufficient to prove that
	\begin{align}\label{vgl}
		\left( \left \lfloor \frac{3}{2}x\right \rfloor-x\right)\sqrt[4]{2}x^2> x-\left \lfloor \frac{3}{2}x\right \rfloor x+\frac{1}{2}\left \lfloor \frac{3}{2}x\right \rfloor\left(\left \lfloor \frac{3}{2}x\right \rfloor-1\right)\left(\frac{x}{2}+1\right) 
	\end{align}
	for all admissible $x$. We denote $\frac{3}{2}x-\left \lfloor \frac{3}{2}x\right \rfloor$ by $\varepsilon$. Then, $0\leq\varepsilon<1$. We rewrite \eqref{vgl} as
	\begin{align}\label{vgleps}
	&\left(\frac{3}{2}x-\varepsilon-x\right)\sqrt[4]{2}x^2> x-\left(\frac{3}{2}x-\varepsilon\right)x+\frac{1}{2}\left(\frac{3}{2}x-\varepsilon\right)\left(\frac{3}{2}x-\varepsilon-1\right)\left(\frac{x}{2}+1\right)\nonumber\\
	\Leftrightarrow\qquad &-\left(\frac{x+2}{4}\right)\varepsilon^{2} +\left(\frac{(3-4\sqrt[4]{2})x^{2}+x-2}{4}\right)\varepsilon +\frac{(8\sqrt[4]{2}-9)x^{3}+12x^{2}-4x}{16}>0\;.
	\end{align}
	The nontrivial zero of the quadratic function $f(\varepsilon)=-\left(\frac{x+2}{4}\right)\varepsilon^{2} +\left(\frac{(3-4\sqrt[4]{2})x^{2}+x-2}{4}\right)\varepsilon$ is smaller than 1 for any $x$, so $f(\varepsilon)>f(1)$ for any $\varepsilon\in[0,1[$ regardless of $x$. So, to prove \eqref{vgleps}, it is sufficient to prove
	\begin{align*}
		&\left(\frac{1}{2}\sqrt[4]{2}-\frac{9}{16}\right)x^{3} +\left(\frac{3}{2}-\sqrt[4]{2}\right)x^{2} -\frac{1}{4}x-1\geq 0\\
		\Leftrightarrow\qquad&(x-2)\left((8\sqrt[4]{2}- 9)x^{2}+6x+8 \right)\geq 0\;,
	\end{align*}
	which is clearly true for $x\geq 2$.
\end{proof}

\begin{lemma}\label{lemmainequality}
	If $2\leq x\leq \frac{1}{\sqrt[8]{2}}f(q,n,k)$ and $n\geq2k+2$ and $q\geq3$, then
	\begin{align*}
		\frac{x-1}{\frac{3}{2} x-2}\qbin{n-k-1}{k} q^{k^2+k} -\left(\frac{3}{2} x-3\right)s'_2 &> x \qbin nk -x \qbin{n-k-1}{k} q^{k^2+k}\quad\text{and}\\
		\frac{x-1}{\frac{3}{2} x-2}\qbin{n-k-1}{k} q^{k^2+k} -\left(\frac{3}{2} x-3\right)s'_2 &> \qbin nk - \qbin{n-k-1}{k} q^{k^2+k}+q^{k+1}\;.   
	\end{align*}
\end{lemma}
\begin{proof} 
	To prove to first inequality we rewrite it using Notation \ref{deltaenC}.
	\begin{align*}
		\frac{x-1}{\frac{3}{2} x-2}\Delta -\left(\frac{3}{2} x-3\right)(xC+(2-x)\Delta+(x-2)W_{\Sigma}) > xC\;.
	\end{align*}
	Using Lemma \ref{LemmaWsigma} we see that it is sufficient to prove
	\begin{align*}
		\frac{x-1}{\frac{3}{2} x-2}\Delta> C\left(\frac{3}{4}x^2+x-3\right)\;.%
	\end{align*}
	From Lemma \ref{lemmadriehoek}, we know that $\frac{\Delta}{C}>\sqrt[4]{2}x^2$. Hence it is sufficient to prove that
	\begin{align*}
		\frac{x-1}{\frac{3}{2} x-2} \sqrt[4]{2}x^2>\left(\frac{3 }{4}x^2+ x-3\right) \quad\Leftrightarrow \quad \left(\sqrt[4]{2}-\frac{9}{8}\right)x^3 -\sqrt[4]{2}x^2+\frac{13}{2}x-6>0\;.
	\end{align*}
	Using a computer algebra packet, we find that the last inequality is valid for all $x\geq 2$.
	\par To prove the second inequality for $k\geq2$ it is sufficient to prove that
	\begin{align*}
		&x \qbin nk -x \qbin{n-k-1}{k} q^{k^2+k} > \qbin nk - \qbin{n-k-1}{k} q^{k^2+k}+q^{k+1}\\
		\Leftrightarrow\quad &q^{k+1}<(x-1) \left( \qbin nk -\qbin{n-k-1}{k} q^{k^2+k} \right)=(x-1)\sum_{i=0}^{k}q^{ik}\qbin{n-i-1}{k-1}\;,
	\end{align*}
	whereby we applied repeatedly the left equality in \eqref{qpascal}. We immediately see that
	\begin{align*}
		(x-1)\sum_{i=0}^{k}q^{ik}\qbin{n-i-1}{k-1}> q^{k^{2}}\qbin{n-k-1}{k-1}>q^{(n-k)(k-1)+k}>q^{2k+2}>q^{k+1}\;.
	\end{align*}
	For $k=1$ we prove the second inequality directly. Note that $s'_{2}=x+2q$. The inequality reduces to
	\begin{align}\label{ongl2}
		&\frac{x-1}{\frac{3}{2} x-2}\cdot\frac{q^{n-2}-1}{q-1}q^{2} -\left(\frac{3}{2} x-3\right)(x+2q) >q^{2}+q+1\nonumber\\
		\Leftrightarrow\quad&\frac{x-1}{\frac{3}{2} x-2}\cdot\frac{q^{n-2}-1}{q-1}q^{2} >\frac{3}{2}x^{2}+3(q-1)x+q^{2}-5q+1\;. 
	\end{align}
	Recall that $2\leq x\leq \frac{1}{\sqrt[8]{2}}f(q,n,1)=\frac{1}{\sqrt[8]{2}}q^{\frac{n-5}{2}}\sqrt{q^{3}-1}<q^{\frac{n-2}{2}}$ We look at the left hand side of \eqref{ongl2} and find
	\begin{multline*}
	\frac{x-1}{\frac{3}{2} x-2}\cdot\frac{q^{n-2}-1}{q-1}q^{2}=\left(\frac{2}{3}+\frac{2}{3(3x-4)}\right)\frac{q^{n-2}-1}{q-1}q^{2}>\left(\frac{2}{3}+\frac{2}{9(x-1)}\right)\frac{q^{n-2}-1}{q-1}q^{2}\\>\frac{2}{3}\frac{q^{n-2}-1}{q-1}q^{2}+\frac{2}{9\left(q^{\frac{n-2}{2}}-1\right)}\frac{q^{n-2}-1}{q-1}(q^{2}-1)=\frac{2}{3}\frac{q^{n-2}-1}{q-1}q^{2}+\frac{2}{9}\left(q^{\frac{n-2}{2}}+1\right)(q+1)\;. 
	\end{multline*}
	For the right hand side of \eqref{ongl2} we find that
	\begin{align*}
		\frac{3}{2}x^{2}+3(q-1)x+q^{2}-5q+1&<\frac{3}{2\sqrt[4]{2}}q^{n-5}\left(q^{3}-1\right)+3(q-1)q^{\frac{n-2}{2}}+q^{2}-5q+1\\&<\frac{3}{2}q^{n-5}\left(q^{3}-1\right)+3(q-1)q^{\frac{n-2}{2}}+q^{2}-5q+1\;.
	\end{align*}
	So, to prove \eqref{ongl2} it is sufficient to prove that
	\begin{align}\label{ongl2bis}
		&\frac{2}{3}\frac{q^{n-2}-1}{q-1}q^{2}+\frac{2}{9}\left(q^{\frac{n-2}{2}}+1\right)(q+1)\geq\frac{3}{2}q^{n-5}\left(q^{3}-1\right)+3(q-1)q^{\frac{n-2}{2}}+q^{2}-5q+1\nonumber\\
		\Leftrightarrow\quad&\frac{2}{3}q^{n-1}-\frac{5}{6}q^{n-2}+\frac{2}{3}\frac{q^{n-4}-1}{q-1}q^{2}+\frac{3}{2}q^{n-5}-q^{\frac{n-2}{2}}\left(\frac{25}{9}q-\frac{29}{9}\right)-q^{2}+\frac{47}{9}q-\frac{7}{9}\geq 0\;.
	\end{align}
	For $n=4,5$ we can check this to be true for all $q\geq3$ using computer algebra software. For $n\geq6$ we rewrite \eqref{ongl2bis} as follows:
	\begin{align*}
		\frac{5}{18}(q-3)q^{n-2}+\frac{q^{\frac{n}{2}}}{18}\left(7q^{\frac{n-2}{2}}-50\right)+\frac{2}{3}\frac{q^{n-4}-1}{q-1}q^{2}+\left(\frac{29}{9}q^{\frac{n-2}{2}}-q^{2}\right)+\frac{47}{9}q+\left(\frac{3}{2}q^{n-5}-\frac{7}{9}\right)\geq 0\;.
	\end{align*}
	Here each of the terms in the left hand side is positive for $q\geq3$ since $n\geq6$, which proves the second inequality in the statement for $k=1$.
\end{proof}

\begin{lemma}\label{lemmaLcontainspoint-pencil}
If $\mathcal{L}$ is a Cameron-Liebler set of $k$-spaces in $\PG(n,q)$, $n\geq3k+2$ and $q\geq3$, with parameter $2\leq x\leq \frac{1}{\sqrt[8]{2}}f(q,n,k)$, then $\mathcal{L}$ contains a point-pencil.
\end{lemma}
\begin{proof}
Let $\pi$ be a $k$-space in $\mathcal{L}$ and let $c$ be the maximal number of elements of $\mathcal{L}$ that are pairwise disjoint. By Theorem \ref{theodef}(3), there are $(x-1)\qbin{n-k-1}{k} q^{k^2+k}$ $k$-spaces in $\mathcal{L}$ disjoint from $\pi$. Within this collection of $k$-spaces, we find at most $c-1$ spaces $\sigma_1,\sigma_2,\dots, \sigma_{c-1}$ that are pairwise disjoint. By Lemma \ref{lemmaongelijkheidklauspargroterdanmeer}, $c-1 \leq \left \lfloor \frac{3}{2}x\right \rfloor -2$. By the pigeonhole principle, we find an index $i$ so that $\sigma_i$ meets at least $\frac{x-1}{c-1}\qbin{n-k-1}{k} q^{k^2+k}\geq \frac{x-1}{\left \lfloor \frac{3}{2}x\right \rfloor-2}\qbin{n-k-1}{k} q^{k^2+k}$ elements of $\mathcal{L}$ that are skew to $\pi$. We denote this collection of $k$-spaces disjoint from $\pi$ and meeting $\sigma_i$ in at least a point by $\mathcal{F}_{i}$. 
\par Now we want to show that $\mathcal{F}_{i}$ contains a family of pairwise intersecting subspaces. For any $\sigma_j$ with $j \neq i$, we find at most $s'_2$ elements that meet $\sigma_i$ and $\sigma_j$. In this way, we find that there are at least $\frac{x-1}{\left \lfloor \frac{3}{2}x\right \rfloor-2}\qbin{n-k-1}{k} q^{k^2+k} -(c-2)s'_2 \geq \frac{x-1}{\frac{3}{2} x-2}\qbin{n-k-1}{k} q^{k^2+k} -\left(\frac{3}{2} x-3\right)s'_2$ elements of $\mathcal{L}$ that meet $\sigma_i$, are disjoint from $\pi$ and that are disjoint from $\sigma_j$ for all $j\neq i$. We denote this subset of $\mathcal{F}_{i}\subseteq\mathcal{L}$ by $\mathcal{F}'_{i}$. This collection $\mathcal{F}'_{i}$ of $k$-spaces is a set of pairwise intersecting $k$-spaces: if two elements $\alpha, \beta$ in $\mathcal{F}'_{i}$ would be disjoint, then $ (\{ \sigma_1,\dots,\sigma_{c-1}\} \setminus \{\sigma_i\})\cup \{\alpha, \beta,\pi \}$ would be a collection of $c+1$ pairwise disjoint elements of $\mathcal{L}$, which is impossible since we supposed that $c$ is size of the maximal set of pairwise disjoint $k$-space in $\mathcal{L}$. By Lemma \ref{lemmainequality} we have $\frac{x-1}{\frac{3}{2} x-2}\qbin{n-k-1}{k} q^{k^2+k} -\left(\frac{3}{2} x-3\right)s'_2 > \qbin{n}{k}-\qbin{n-k-1}{k} q^{k^2+k}+q^{k+1}$ since $2\leq x\leq \frac{1}{\sqrt[8]{2}}f(q,n,k)$. This implies that $\cap_{F\in \mathcal{F}'_{i}} F$ is not empty by Theorem \ref{theomussche}; let $P$ be a point contained in $\cap_{F\in \mathcal{F}'_{i}} F$. 
We conclude that $\mathcal{F}'_{i}$ is a part of the point-pencil through $P$.
\par We conclude by showing that $\mathcal{L}$ contains the whole point-pencil through $P$. If $\gamma\notin \mathcal{L}$ is a $k$-space through $P$, then $\gamma$ meets at least $\frac{x-1}{\frac{3}{2} x-2}\qbin{n-k-1}{k} q^{k^2+k}-(\frac{3}{2} x-3)s'_2 > x\qbin{n}{k}-x\qbin{n-k-1}{k} q^{k^2+k}$ elements of $\mathcal{F}'_{i} \subseteq \mathcal{L}$, where the inequality follows from Lemma \ref{lemmainequality}. This contradicts Theorem \ref{theodef}(3).
\end{proof}

\begin{theorem}
	There are no Cameron-Liebler sets of $k$-spaces in $\PG(n,q)$, $n\geq3k+2$ and $q\geq3$, with parameter $2\leq x\leq \frac{1}{\sqrt[8]{2}}q^{\frac{n}{2}-\frac{k^2}{4}-\frac{3k}{4}-\frac{3}{2}}(q-1)^{\frac{k^2}{4}-\frac{k}{4}+\frac{1}{2}}\sqrt{q^2+q+1}$.
\end{theorem}
\begin{proof}
	We prove this result using induction on $x$. By Lemma \ref{lemmaLcontainspoint-pencil} we know that $\mathcal{L}$ contains the point-pencil $[P]_k$ through a point $P$. By Lemma \ref{basislemma4}(4), $\mathcal{L} \setminus [P]_k$ is a Cameron-Liebler set of $k$-spaces with parameter $(x-1)$, which by the induction hypothesis (in case $x-1\geq2$) or by Lemma \ref{lemmatussen012} (in case $1<x-1<2$) does not exist, or which is a point-pencil (in case $x-1=1$) by Theorem \ref{xgelijkaanee}. In the former case there is an immediate contradiction; in the latter case $\mathcal{L}$ contains two disjoint point-pencils of $k$-spaces, a contradiction.
\end{proof}

\subsection*{Acknowledgements}
The research of Jozefien D'haeseleer is supported by the FWO (Research Foundation Flanders). 
The authors thank Leo Storme and Ferdinand Ihringer for their remarks and suggestions while writing this article.

\end{document}